\documentclass{article}

\usepackage[utf8]{inputenc}
\usepackage{amstext}
\usepackage{amsfonts}
\usepackage{amsthm}
\usepackage{textcomp}
\usepackage{amssymb}
\usepackage{amscd}
\usepackage{amsmath}
\usepackage{xcolor}
\usepackage{tikz-cd}
\usepackage{caption}
\usepackage{enumitem}
\usepackage{theoremref}
\usepackage{mathtools}
\usepackage{float}
\usepackage{stmaryrd}
\setlist[enumerate]{label=({\arabic*})}

\newtheorem{defn}{Definition}[section]
\newtheorem{thm}[defn]{Theorem}
\newtheorem{exmp}[defn]{Example}
\newtheorem{lem}[defn]{Lemma}
\newtheorem{rmk}[defn]{Remark}

\newtheorem{prop}[defn]{Proposition}
\newtheorem{cor}[defn]{Corollary}
\newtheorem{manualtheoreminner}{}
\newenvironment{manualtheorem}[1]{%
  \renewcommand\themanualtheoreminner{#1}%
  \manualtheoreminner
}{\endmanualtheoreminner}

\newcommand\doublea
 {\ensuremath{\mathrel{\subarray{c}
  \mkern.2mu \mkern2mu\Longrightarrow\\[-.33ex]
 {\underset{(a)}{ \Longarrownot \Longleftarrow}} \endsubarray}}}

\newcommand\doubleb
 {\ensuremath{\mathrel{\subarray{c}
  \mkern.2mu \mkern2mu\Longrightarrow\\[-.33ex]
 {\underset{(b)}{ \Longarrownot \Longleftarrow}} \endsubarray}}}

 \newcommand\doublec
 {\ensuremath{\mathrel{\subarray{c}
  \mkern.2mu \mkern2mu\Longrightarrow\\[-.33ex]
 {\underset{(c)}{ \Longarrownot \Longleftarrow}} \endsubarray}}}

\newcommand{\N}{\mathbb{N}}
\newcommand{\R}{\mathbb{R}}
\newcommand{\Z}{\mathbb{Z}}

\newcommand{\B}{{\mathcal B}}

\DeclareMathOperator*\lowlim{\underline{lim}}

\date{\empty}
\begin{document}
\title{$\mathbb R$- and $\mathbb C$-supercyclicity for some classes of operators}

\author{E. D'Aniello, M. Maiuriello}

\newcommand{\Addresses}{{
  \bigskip
  \footnotesize
  
    M.~Maiuriello,\\
  \textsc{Dipartimento di Scienze Umane,\\ Universit\`a degli Studi  ``Link Campus University",\\
  Via del Casale San Pio V, 44, 00165 Roma, ITALIA}\\
  \textit{E-mail address: \em m.maiuriello@unilink.it} \\
  
   E.~D'Aniello,\\
  \textsc{Dipartimento di Matematica e Fisica,\\ Universit\`a degli Studi della Campania ``Luigi Vanvitelli",\\
  Viale Lincoln n. 5, 81100 Caserta, ITALIA}\\
  \textit{E-mail address: \em emma.daniello@unicampania.it} 

}}

\maketitle

\begin{abstract}
In the present paper we investigate different variants of supercyclicity, precisely $\mathbb R^+$-, $\mathbb R$- and $\mathbb C$-supercyclicity in the context of composition operators. 
We characterize $\R$-supercyclic composition operators on $L^p$, $1 \leq p < \infty$. Then, we turn our attention to dissipative composition operators, and we show that $\R$- and 
$\mathbb C$-supercyclicity are equivalent notions in this setting and they have a ``shift-like'' characterization.
\end{abstract}

\let\thefootnote\relax\footnote{\date{\empty} 
}

\section{Introduction}

Nowadays, when considering the term ``chaos" in the context of dynamical systems and operator theory, the belief that it is intrinsically linked to non-linearity is outdated.
The investigations of one of its main ingredients, hypercyclicity, have thoroughly provided many examples, even quite natural, of linear dynamical systems exhibiting chaos. 
However, in the realm of finite-dimensional spaces the existence of a hypercyclic linear operator is precluded and, as initially observed by G. D. Birkhoff \cite{Birkhoff}, this effect 
only becomes visible when dealing with the infinite-dimensional case. Starting from the seemingly plain definition of a hypercyclic operator, that is, an operator with a dense orbit, 
the theory has developed into a very active research area leading, in addition, to the analysis of several related phenomena such as  supercyclicity, frequent hypercyclicity, topological mixing and Li-Yorke chaos. \\
Anyone with an even modest familiarity with linear dynamics is aware that these notions have been extensively studied, and sometimes characterized, in the large and versatile 
family of composition operators, known to be used in many applications like economic ones \cite{leventides2022analysis}, and in this context (more generally, in applied sciences) they are referred to as Koopman operators. This class includes the weighted shifts renowned for being an indispensable tool for constructing examples and understanding complex dynamics.\\
 Let us briefly recall the composition operators on $L^p(X)$, $1\leq p < \infty$. Given a $\sigma$-finite measure space $(X,{\mathcal B},\mu)$ and a bijective, bimeasurable 
 transformation $f:X \rightarrow X$ with both the Radon-Nikodym derivatives of $\mu \circ f$  and $\mu \circ f^{-1}$ with respect to $\mu$ bounded from below, then the operator 
 $T_f:L^p(X) \rightarrow L^p(X)$, defined by $T_f (\varphi) = \varphi \circ f $, is a well-defined bounded invertible linear operator on $L^p(X)$, known as {\em{composition operator}}. 
 The quintuple  $(X,{\mathcal B},\mu, f, T_f)$ is referred to as a {\em{composition dynamical system}}.
For a general composition dynamical system as above, necessary and sufficient conditions on the transformation $f$ that guarantee hypercyclicity and mixing for $T_f$ are 
provided in \cite{BADP} and, in \cite{BDP}, a characterization for $T_f$ being Li-Yorke chaotic is given. These are very general results and, as a specific case, when $X = \Z$ and 
$f$ is the $+1$-map, they yield well-known characterizations of hypercyclicity, mixing and Li-Yorke chaos for weighted backward shifts. We recall that, given a bounded sequence 
of complex numbers $\{w_i\}_{i \in \Z}$,  a {\em bilateral weighted backward shift with weights $\{w_i\}_{i \in \Z}$} is the bounded linear operator $B_w: \ell^p(\Z) \rightarrow \ell^p(\Z)$ 
defined by $B_w({\bold x})(i) = w_{i+1}x_{i+1}$ and, moreover, $B_w$ is invertible when $\{w_i\}_{i \in \Z}$ is bounded away from zero.
When considering ${\mathbb N}$ instead of $\Z$, the weighted backward shift is called {\em unilateral}. We refer the reader to the books \cite{BayartMatheron, GE} and 
to the more recent papers \cite{Bayartruzsa, CHARPENTIER20164443, DAnielloMaiurielloSpectrum, GrossePapathanasiou, Maiuriello}, where many additional references can be found.
In addition to the aforementioned characterizations, also supercyclicity, frequent hypercyclicity and chaos are completely understood for weighted shifts (see \cite{GE}), while, 
unfortunately, they have not yet been characterized for composition operators in the general context. To this aim, in \cite{DAnielloDarjiMaiuriello, Darjipires} the authors introduced
 important instruments and ideas that played a key role in addressing such a gap in the literature. More precisely, the authors introduced the notions of dissipativity and bounded 
 distortion in the context of composition operators, which occur, respectively, when there exists $W \in \B, 0 < \mu (W) < \infty$ such that
\[ X = \dot {\cup} _{k\in \Z}f^k (W) \tag{D} \]
and when, in addition, there exists $K>0$ such that
\[ \dfrac{1}{K} \mu(f^k(W))\mu(B) \leq \mu(f^k (B))\mu (W) \leq K \mu(f^k(W))\mu(B) \tag{BD} \]
for all $k \in \mathbb Z$ and  $B \subseteq W, \mu(B)>0$.
In \cite{Darjipires}, necessary and sufficient condition on $f$ to guarantee chaos and frequent hypercyclicity for dissipative $T_f$, with the bounded distortion, are provided. 
In the same context, not only chaotic properties but even hyperbolic ones, like generalized hyperbolicity and shadowing, are characterized in \cite{DAnielloDarjiMaiuriello}. After that, additional 
truths emerged in the dissipative context. Results from \cite{DAnielloDarjiMaiuriello2} reveal that, fixing $\nu, g$ and $\{w_i\}_{i \in \Z}$ properly, then the two systems $(\Z, {\mathcal{P}}(\Z) ,\nu, g, B_w)$ 
and $(X,{\mathcal B},\mu, f, T_f)$ are indistinguishable from a dynamical perspective: precisely, they share common dynamical properties including hypercyclicity, topological mixing, chaos, 
Li-Yorke chaos, (uniform) expansivity, shadowing and (generalized) hyperbolicity, and, for this reason, the authors coined the term ``shift-like'' referring to the behavior of composition operators in the dissipative setting.

From this overview, it emerges that the aforementioned results are not, as of today, yet known for the various variants of supercyclicity. 
Specifically, while $\mathbb C$-, $\R$- and $\R^+$-supercyclicity are characterized for weighted shifts, it seems not to be so yet for 
composition operators, neither in the general context nor in the dissipative one. Moreover, in the dissipative context, 
it is not yet established if this property falls among those shared by $B_w$ and $T_f$, i.e. if $T_f$ has a shifts-like behavior concerning supercyclicity. 
The purpose of this paper is to fix this gap in the literature. The concept of $\mathbb C$-supercyclic operators was introduced by Hilden and Wallen in \cite{hilden} and, since then, a multitude of variants have been studied. 
In \cite{herzog1992linear}, Herzog proved that every real or complex, separable, infinite dimensional Banach space supports a $\mathbb C$-supercyclic operator.  Section 2 recalls some preliminaries. In Section 3, we provide a very general characterization 
of $\R$-supercyclicity for composition operators. Section 4 focuses on the dissipative context: here, auxiliary results are presented that, together with the 
already known techniques from \cite{DAnielloDarjiMaiuriello2}, will enable us to prove the equivalence between $\R$- and $\mathbb C$-supercyclicity and, in addition, to show the shift-like 
behavior with respect to these properties.

\section{Preliminaries}
In the sequel $X$ will be a separable metric space.
We outline the general theory of cyclic, hypercyclic and supercyclic operators, recalling definitions and summarizing classical results from the literature.
 
\begin{defn}
Let $T : X \rightarrow X $ be an operator. A vector $x \in X$ is called 
\begin{itemize}
\item{{\em{cyclic}} for $T$ if the linear span of its orbit, i.e. $\text{span}\{T^nx ; n\geq 0\}$, is dense in $X$;}
\item{{\em{${\mathbb C}$-supercyclic}} (or, simply, {\em{supercyclic}}) for $T$ if its projective orbit, i.e. $\{ \lambda T^nx ; n \geq 0, \lambda \in \mathbb C\}$ is dense in $X$;}
\item{{\em{${\mathbb R}$-supercyclic}} for $T$ if $\{ \lambda T^nx ; n \geq 0, \lambda \in \mathbb R\}$ is dense in $X$;}
\item{{\em{${\mathbb R}^{+}$-supercyclic}} for $T$ if $\{ \lambda T^nx ; n \geq 0, \lambda \in {\mathbb R}^{+}\}$ is dense in $X$;}
\item{ {\em{hypercyclic}} for $T$ if the orbit itself, i.e. $ \{ T^nx ; n \geq 0 \}$ is dense in $X$.}
\end{itemize}
\end{defn}

Of course, these notions make sense only if the space $X$ is separable. Hence, as already mentioned above, from now on, $X$ is a separable Banach space 
and $T$ is a bounded linear operator from $X$ to itself. 
The set of cyclic, supercyclic and hypercyclic vectors for $T$ is denoted by $C(T), SC(T)$ and $HC(T)$, respectively. Operators admitting a cyclic, a supercyclic and 
a hypercyclic vector are called {\em{cyclic, supercyclic}} and {\em{hypercyclic operators}}, respectively. Clearly, supercyclicity falls between hypercyclicity and cyclicity, 
as the following diagram shows.\\

\vspace{0.03 cm}
\noindent
{\small{
\text{Hypercyclicity} \doublea \text{${\mathbb R}^{+}$-supercyclicity}  $\Leftrightarrow$ \text{${\mathbb R}$-supercyclicity}  \doubleb \text{Supercyclicity}  \doublec  \text{Cyclicity} \\
}}
\vspace{0.02 cm}

\noindent
Note that unilateral unweighted backward shifts are always cyclic \cite{halmos}, but they do not have supercyclic vectors \cite[page 564(b)]{hilden}, hence $(c)$ cannot be reverted. Moreover, 
regarding $(b)$, in \cite[Remark 2.1]{bermudez2002c} it is constructed a supercyclic operator not being $\R-$supercyclic, and, in \cite[Proposition 4.1]{Salas}, an example of a supercyclic 
weighted shift having no hypercyclic multiples (and, therefore, not hypercyclic \cite{S}) is provided, showing that the inverse of $(a)$ does not hold. The equivalence between $\R$- and 
$\R^+$-supercyclicity is showed in \cite[Theorem 2.1]{bermudez2002c}\\

The following well-known result was proved by Birkhoff in \cite{birkhoff1922}. 

\begin{manualtheorem}{Birkhoff Transitivity Theorem}
Let $T \in  L(X)$. The following are equivalent:
\begin{itemize}
\item[(i)] $T$ is hypercyclic;
\item[(ii)] $T$ is topologically transitive; that is, for each pair of non-empty open subsets $U,V$ of $X$ there exists $n\in \mathbb N$ such that $T^n(U)\cap V\neq \emptyset.$
\end{itemize}
\end{manualtheorem}

Next, we provide the analogue of the Birkhoff Transitivity Theorem for supercyclicity and ${\mathbb R}^+$-supercyclicity.

\begin{thm} \cite[Theorem 1.12]{BayartMatheron} \label{super} Let $T \in  L(X)$. The following are equivalent:
\begin{itemize}
\item[(i)] $T$ is $\mathbb C$-supercyclic;
\item[(ii)] for each pair of non-empty open subsets $U,V$ of $X$ there exists $n\in \mathbb N$ and $\lambda \in \mathbb C$ such that $\lambda T^n(U)\cap V\neq \emptyset.$
\end{itemize}
\end{thm}

\begin{thm} \label{R-super}
Let $T \in  L(X)$. The following are equivalent:
\begin{itemize}
\item[(i)] $T$ is ${\mathbb R}$-supercyclic;
\item[(ii)] for each pair of non-empty open subsets $U,V$ of $X$ there exist $n \in \mathbb N$ and $\lambda \in {\mathbb R}^{+}$ such that $\lambda T^n(U)\cap V\neq \emptyset.$
\end{itemize}
\end{thm}
\begin{proof}
$(i) \Rightarrow (ii)$. By hypothesis, there exists $x \in X$ such that $\{ \lambda T^nx ; n \geq 0, \lambda \in {\mathbb R}^{+}\}$ is dense in $X$. Note that we can consider 
$\lambda \in  {\mathbb R}^{+}$ by the above mentioned equivalence between $\R$- and 
$\R^+$-supercyclicity. Consider $U$ and $V$ non-empty open subsets of $X$. 
Note that, by the density of the orbit above, there exist $n \in \N$ and $\lambda \in \R^+$ such that $\lambda T^n(x) \in U$, and there exist $n' \in \N$, $n'> n$ and ${\lambda}' \in \R^+$ such 
that ${\lambda '} T^{n'}(x) \in V$. Therefore, $\frac{{\lambda '}}{\lambda} T^{n' - n}(U)$ intersects $V$.   \\
To show $(ii) \Rightarrow (i)$, consider, the space being separable, $\{B_k\}_{k \in \N}$ an enumeration of the open balls with rational radii and centered on a dense countable family. Let $U$ be 
any open set. By assumption, there exist $n_{1} \in \mathbb N$ and ${\lambda_{1}} \in {\mathbb R}^{+}$ such that $\frac{1}{{\lambda}_{1}} T^{-n_{1}}(B_{1})\cap U \neq \emptyset$. 
Since $\frac{1}{{\lambda}_{1}} T^{-n_{1}}(B_{1})\cap U$ is an open set, it must contain some open ball, let us call it $U_{1}$, of radius  less than $\frac{1}{2}$. By hypothesis again, 
there exist $n_{2} \in \mathbb N$ and ${\lambda_{2}} \in {\mathbb R}^{+}$ such that $\frac{1}{{\lambda}_{2}} T^{-n_{2}}(B_{2})\cap U_{1} \neq \emptyset$. Since 
$\frac{1}{{\lambda}_{2}} T^{-n_{2}}(B_{2})\cap U_{1}$ is an open set, it must contain some open ball, call it $U_{2}$, of radius less than $ \frac{1}{4}$, and so on. 
Hence, the intersection $\cap_{n \in \N} U_{n}$ is a single point $x$ satisfying ${\lambda}_{k}T^{n_{k}}(x) \in B_{k}$ for all $k \in \N$, meaning that $x$ is a ${\mathbb R}^{+}$-supercyclic vector for $T$.
\end{proof}

 The above two results, together with the following criteria, are the main tools used to show that an operator possesses certain forms of cyclicity.

\begin{manualtheorem}{Supercyclicity Criterion} \label{GS} Given the operator $T$, if there are dense subsets $X_0,Y_0 \subset X$, a sequence $\{n_k\} \rightarrow \infty$, and 
$S:Y_0 \rightarrow Y_0$ such that
\begin{itemize}
\item[(a)] if $x\in X_0$ and $y\in Y_0$, then $\Vert T^{n_k}x\Vert \Vert S^{n_k}y\Vert \rightarrow 0$ as $\{n_k\} \rightarrow \infty$
\item[(b)]  $TSy=y$, for any $y\in Y_0$
\end{itemize}
then $T$ is $\mathbb C$-supercyclic.
\end{manualtheorem}

\begin{manualtheorem}{Hypercyclicity Criterion} \label{GS} Given the operator $T$, if there are dense subsets $X_0,Y_0 \subset X$, a sequence $\{n_k\} \rightarrow \infty$, and $S:Y_0 \rightarrow Y_0$ such that
\begin{itemize}
\item[(a)] $T^{n_k}x\rightarrow 0$, for any $x\in X_0$
\item[(b)] $S^{n_k}y \rightarrow 0$, for any $y\in Y_0$
\item[(c)] $TSy=y$, for any $y\in Y_0$.
\end{itemize}
then $T$ is hypercyclic.
\end{manualtheorem}

The above recalled Supercyclicity Criterion was developed by Salas in \cite{Salas} and it is strictly more general than 
the one previously provided in \cite{Kitai}. The first version of the Hypercyclicity Criterion was given by Kitai in \cite{Kitai}, while the one above was independently rediscovered by Gethner and 
Shapiro in \cite{Shapiro}. For some operators, the above conditions are both necessary and sufficient: for instance, a bilateral weighted backward shift is hypercyclic if and only if it satisfies 
the Hypercyclicity Criterion  \cite{leon2001spectral} and the same holds for supercyclicity together with the Supercyclicity Criterion \cite[Corollary 3.2]{Salas}. In addition, since, in general, 
the Supercyclicity Criterion implies the $\R$-supercyclicity \cite[Remark 2.2]{bermudez2002c}, then it follows that for bilateral weighted shifts these three supercyclicity notions 
(${\mathbb R}^{+}$-, ${\mathbb R}$-  and  $\mathbb C$-supercyclicity) are equivalent. \\

We recall that two linear operators $T:X \rightarrow X$ and $S:Y \rightarrow Y$ are said to be {\em topologically semi-conjugate} if there exists a linear, bounded, surjective map $\Pi: X \rightarrow Y$, 
called a {\em factor map,} for which $\Pi \circ T=S\circ \Pi$. In such case, $S$ is called a {\em factor} of $T$. In particular, if $\Pi$ is a homeomorphism, then {\em topological conjugacy} is achieved, 
which is the standard equivalence used in dynamical systems theory to say that two operators have the same dynamics. It was already observed, and it is easy to verify, that cyclicity, hypercyclicity, 
both ${\mathbb R}$- and ${\mathbb C}$-supercyclicity are preserved by factors \cite{BayartMatheron, GE}, and we shall make extensive use of this in Section 4.

\section{$\mathbb R$-supercyclicity for composition operators}

Our first result is a characterization of ${\mathbb R}$-supercyclic composition operators in the general setting. The proof follows ideas given in \cite{Kalmes} and \cite[Theorem 1.1]{BADP} for the hypercyclic case. In the sequel, $X$  
is a locally compact, $\sigma$-compact and second countable Hausdorff space, ${\mathcal B}$ is the Borel $\sigma$-algebra over $X$, and $\mu$ is a locally finite Borel measure (and, therefore, it is $\sigma$-finite as $X$ is 
assumed to be $\sigma$-compact). We recall that, given a set $A$ and $\delta >0$,  $B_{\delta}(\chi_A)$ is the open ball centered in $\chi_A$ with radius $\delta$.
 
\begin{thm}[Characterization of $\R$-supercyclicity - General case] \label{generalcase}
Let $(X,{\mathcal B},\mu, f, T_f)$ be a composition dynamical system. The composition operator $T_f$ is ${\mathbb R}$-supercyclic if and only if for all $\epsilon >0$, for all $B \in \mathcal B$ of finite measure, there exist $B' \subseteq B$, $k \geq 1$ and $\lambda >0$, 
such that \[ \mu(B \setminus B')< \epsilon, \,\, \,\,  \mu(f^{k}(B'))<  \lambda^p\epsilon \text{  and  } \mu(f^{-k}(B'))< \lambda^{-p} \epsilon. \]
\end{thm}
\begin{proof}
$(\Rightarrow )$. Let $B \in \mathcal B$ with $0 < \mu(B) < \infty$ be fixed and let $0< \epsilon < \frac{1}{2^{p}}$.   
By hypothesis, $T_f$ is ${\mathbb R}$-supercyclic and then, by Theorem \ref{R-super}, there exist $k \in \mathbb N$ and $\lambda \in {\mathbb R}^{+}$ such that 
\[\lambda T_f^k(B_{{\left(\frac{\epsilon}{2}\right)}^{\frac{2}{p}}}( - \chi_B)) \cap (B_{{\left( \frac{\epsilon}{2}\right)}^{\frac{2}{p}}}(\chi_B))\neq \emptyset, \tag{3.1} \label{A}.\]
Hence, there exists $\varphi \in L^{p}(X)$ such that 
\[ {\Vert \varphi \circ f^{-k} + \lambda \chi_{B} \Vert}_{p}^{p} < {\lambda}^{p} {\left(\frac{\epsilon}{2}\right)}^{2}  \hspace{0.3cm} \&  \hspace{0.3cm}  {\Vert \varphi - \chi_{B} \Vert}_{p}^{p} < {\left(\frac{\epsilon}{2}\right)}^{2}.\]

By, eventually, replacing $\varphi$ with $Re(\varphi)$, that is the real part of $\varphi$, we can, without loss of generality assume $\varphi$ real-valued. \\
Recall that, given a real-valued function $g$, it is $g=g^+-g^-$, where $g^{+}  := \max \{g, 0\}$, $g^{-}  := \max \{-g, 0\}$. In particular, the mapping $L^{p}(X, {\mathbb R}) \rightarrow L^{p}(X, {\mathbb R})$, $g \mapsto g^{+}$, satisfies ${\Vert {(g + h)}^{+}\Vert}_{p} \leq 
 {\Vert {g}^{+} + h^{+} \Vert}_{p}$, and it commutes with $T_f$, that is ${(T_{f}(g))}^{+} = T_{f}(g^{+})$. \\
 Hence, we have
  \begin{eqnarray*}
{\Vert {(T_{f}^{-k}(\varphi))}^{+} \Vert}_{p} &=&  {\Vert {(T_{f}^{-k}(\varphi) - (- \lambda \chi_{B}) + (- \lambda \chi_{B}))}^{+} \Vert}_{p}\\
 &  \leq  &  {\Vert {(T_{f}^{-k}(\varphi) - (- \lambda \chi_{B}))}^{+} \Vert}_{p}  +  {\Vert {(- \lambda \chi_{B})}^{+} \Vert}_{p} \\
& = &   {\Vert {(T_{f}^{-k}(\varphi) - (- \lambda \chi_{B}))}^{+}\Vert}_{p}  \leq {\Vert T_{f}^{-k}(\varphi)  +  \lambda \chi_{B}\Vert}_{p} \\
& < &  \lambda {\left(\frac{\epsilon}{2}\right)}^{\frac{2}{p}}, 
\end{eqnarray*}
that is 
\[{\Vert {(T_{f}^{-k}(\varphi))}^{+} \Vert}_{p} <  \lambda {\left(\frac{\epsilon}{2}\right)}^{\frac{2}{p}}. \tag{3.2} \label{B}\] 
Moreover,
\begin{eqnarray*}
  {\Vert {\varphi}^{-} \Vert}_{p} &=&  {\Vert {(-\varphi)}^{+} \Vert}_{p}  = {\Vert {(\chi_{B} - \varphi - \chi_{B})}^{+} \Vert}_{p}\\
& \leq &    {\Vert (\chi_{B} - \varphi )^+\Vert}_{p} + {\Vert {( - \chi_{B} )}^{+} \Vert}_{p}  = {\Vert (\chi_{B} - \varphi )^+\Vert}_{p} \leq  {\Vert \chi_{B} - \varphi \Vert}_{p}\\
&  < &  {\left(\frac{\epsilon}{2}\right)}^{\frac{2}{p}},
\end{eqnarray*}
that is 
\[ {\Vert {\varphi}^{-} \Vert}_{p} <  {\left(\frac{\epsilon}{2}\right)}^{\frac{2}{p}}  \tag{3.3}. \label{C}\] 
Let $B_{1}:=  B \cap \{ {\vert 1  - \varphi \vert}^{p} > \frac{\epsilon}{2}\}$ and $B_2 := B \cap \{{\vert \lambda + \varphi \circ f^{-k} \vert}^{p} > \lambda^{p} \frac{\epsilon}{2}\}.$
Then, as 
\begin{eqnarray*}
\frac{\epsilon}{2} \mu(B_{1})  &  = & \int_{B_{1}} \frac{\epsilon}{2} d \mu < \int_{B_{1}} {\vert 1 - \varphi \vert}^{p} d \mu  =  \int_{B_{1}} {\vert \chi_{B} - \varphi \vert}^{p} d \mu \\
&\leq& \int_{X} {\vert \chi_{B} - \varphi \vert}^{p} d \mu=  {\Vert \chi_{B} - \varphi \Vert}_{p}^{p} <  {\left (\frac{\epsilon}{2} \right)}^2
\end{eqnarray*}
and, analogously,
 \begin{eqnarray*}
{\lambda}^{p} \frac{\epsilon}{2} \mu(B_{2})  &  = & \int_{B_{2}} {{\lambda}^{p} \frac{\epsilon}{2} }d \mu < \int_{B_{2}} {\vert \lambda + \varphi \circ f^{-k} \vert}^{p}  d \mu  =  \int_{B_{2,}} {\vert \lambda \chi_{B} +  \varphi \circ f^{-k} \vert}^{p} d \mu\\ 
&\leq& \int_{X} {\vert \lambda \chi_{B} + \varphi \circ f^{-k} \vert}^{p} d \mu =  {\Vert \lambda \chi_{B} +  \varphi \circ f^{-k}  \Vert}_{p}^{p} <  {\lambda}^{p} {\left (\frac{\epsilon}{2} \right)}^2,\\
\end{eqnarray*}
 it is $\mu(B_{1}) < \frac{\epsilon}{2}$ and $\mu(B_{2}) < \frac{\epsilon}{2}$.  \\
Defining \[B^{'} := B \cap\left  \{ {\vert 1  - \varphi \vert}^{p} \leq \frac{\epsilon}{2} \right \} \cap \left \{{\vert \lambda + \varphi \circ f^{-k} \vert}^{p} \leq {\lambda}^{p} \frac{\epsilon}{2}\right \},\] then 
\[\mu(B \setminus B^{'}) \leq \mu(B_{1}) + \mu(B_{2}) < \frac{\epsilon}{2} + \frac{\epsilon}{2} =\epsilon.\] 

\noindent
Note that  ${\varphi}_{|B'}  >0$ and $(\varphi \circ f^{-k})_{|B'} <0$: considering $x \in  B'$, it is, respectively, $\varphi(x)  = 1 + \varphi(x) - 1  \geq 1 - \vert 1 - \varphi(x) \vert  \geq 1 - {\left (\frac{\epsilon}{2} \right)}^{\frac{1}{p}} >0,$ and $(\varphi \circ f^{-k})(x)   =  \lambda + (\varphi \circ f^{-k})(x) - \lambda   \leq   \vert (\varphi \circ f^{-k})(x) + \lambda \vert - \lambda \leq   \lambda \left [ {\left (\frac{\epsilon}{2} \right)}^{\frac{1}{p}}-1 \right ] < 0.$ \\
Hence, 
\begin{eqnarray*}
\frac{1}{2^p}\mu(f^{k}(B')) & \leq & {\left ( 1 -\left ({\frac{\epsilon}{2}}\right )^{\frac{1}{p}} \right )}^{p} \mu(f^{k}(B'))  = \int_{B'} {\left ( 1 - \left ( {\frac{\epsilon}{2}}\right )^{\frac{1}{p}}\right )}^{p} d \mu f^{k} \\
& \leq & \int_{B'} (\varphi^+)^p d \mu f^{k} =  \int_{f^k(B')} ((\varphi \circ f^{-k})^+)^p d \mu\\
&  \leq &  \int_{X} ( (\varphi \circ f^{-k})^+)^{p} d \mu  =  {\Vert (T_f^{-k}\varphi)^{+}\Vert}_{p}^{p} \\
&  < &  \lambda^p {\left( \frac{\epsilon}{2} \right)}^{2} \hspace{0.5cm} \text{     (by (\ref{B}))}
\end{eqnarray*}
and, furthermore, 
\begin{eqnarray*}
\frac{\lambda^p}{2^p} \mu(f^{-k}(B')) & \leq & \lambda^p \left ( 1 - {\left (\frac{\epsilon}{2} \right)}^{\frac{1}{p}}\right )^p \mu(f^{-k}(B'))   =   \int_{B'} \lambda^p \left ( 1 - {\left (\frac{\epsilon}{2} \right)}^{\frac{1}{p}}\right )^p d \mu f^{-k} \\
&\leq& \int_{B'} (-\varphi \circ f^{-k})^{p} d \mu f^{-k} = \int_{B'}  ((\varphi \circ f^{-k})^-)^{p} d \mu f^{-k}\\
&  = &  \int_{f^{-k}(B')} (\varphi^-)^{p} d \mu  \leq   \int_{X} (\varphi^- )^{p} d \mu  \leq  {\Vert {\varphi}^{-}  \Vert}_{p}^{p} \\
&  <  &  {\left( \frac{\epsilon}{2} \right)}^{2} \hspace{0.5cm} \text{     (by (\ref{C}))}
\end{eqnarray*}
so that 
\[\mu(f^{k}(B^{'})) < \lambda^p 2^{p} {\left( \frac{\epsilon}{2} \right)}^{2} < \lambda^p \epsilon, \ \ \text{ and } \ \  \mu(f^{-k}(B^{'})) < {\lambda}^{-p} 2^p  {\left( \frac{\epsilon}{2} \right)}^{2}  < {\lambda}^{-p} \epsilon,\]
implying the thesis.

$(\Leftarrow )$.We apply Theorem \ref{R-super}. Let $U$ and $V$ be non-empty open subsets of $L^{p}(X)$. We need to find $n \in \mathbb N$ and $\lambda \in {\mathbb R}^{+}$ such that $\lambda T^n(U)\cap V\neq \emptyset.$ Let $g$ and $h$ be two continuous functions with compact support in $U$ and $V$, respectively. Let $B$ be 
a compact subset of $X$ containing both the support of $g$ and of $h$.  Let, for each $n \in {\mathbb N}$, ${\epsilon}_{n} \in (0, \min\{\frac{1}{2^n}, \frac{1}{2^p}\})$. By hypothesis, there exist $B_{n} \subseteq B$, $k_{n} \geq 1$, $\lambda_{n} >0$, 
such that 
\[ \mu(B \setminus B_{n})< \epsilon_{n}, \,\, \,\,  \mu(f^{k_{n}}(B_{n}))<  {(\lambda_{n})}^p\epsilon_{n} \text{  and  }  \mu(f^{-k_{n}}(B_{n}))< {(\lambda_{n})}^{-p} \epsilon_{n}, \]
that is the same as 
\[ \mu(B \setminus B_{n})< \epsilon_{n}, \,\, \,\,  {(\lambda_{n})}^{-p} \mu(f^{k_{n}}(B_{n}))<  \epsilon_{n} \text{  and  }  {(\lambda_{n})}^{p} \mu(f^{-k_{n}}(B_{n}))<  \epsilon_{n}. \]

Define, for every $n \in {\mathbb N}$, 
\[v_{n} := g \chi_{B_{n}} + {(\lambda_{n})}^{-1} (h \circ f^{-k_n}) \chi_{f^{k_{n}}(B_{n})}\]
which belong to  $L^p(X)$ since they are measurable, bounded and different from $0$ at most on a subset of $B \cup f^{k_{n}}(B)$. 
Then 
\begin{eqnarray*}
{\Vert v_{n} - g \Vert}_{p}^{p} & = &  {\Vert g \chi_{B_{n}} + {(\lambda_{n})}^{-1} (h \circ f^{-k_n}) \chi_{f^{k_{n}}(B_{n})} - g \Vert}_{p}^{p} \\
& \leq & 2^{p} \left [{\Vert g \Vert}_{\infty}^{p} \mu(B \setminus B_{n}) + {(\lambda_{n})}^{-p} {\Vert h \Vert}_{\infty}^{p} \mu(f^{k_{n}}(B_{n})) \right ]\\
& = & 2^{p} \left [{\Vert g \Vert}_{\infty}^{p} \mu(B \setminus B_{n}) + {\Vert h \Vert}_{\infty}^{p} {(\lambda_{n})}^{-p}  \mu(f^{k_{n}}(B_{n})) \right ]\\
& < & 2^{p} \left [{\Vert g \Vert}_{\infty}^{p} \epsilon_{n} + {\Vert h \Vert}_{\infty}^{p} \epsilon_{n} \right ],\\
\end{eqnarray*}
that is, 
\[\lim_{n \rightarrow \infty} {\Vert v_{n} - g \Vert}_{p} = 0.\]
Moreover, 
\begin{eqnarray*}
{\Vert {\lambda_{n}} T_{f}^{k_{n}} v_{n} - h \Vert}_{p}^{p} & = & {\Vert {\lambda_{n}} T_{f}^{k_{n}} \left (g \chi_{B_{n}} + {(\lambda_{n})}^{-1} (h \circ f^{-k_n}) \chi_{f^{k_{n}}(B_{n})} \right) - h \Vert}_{p}^{p}\\
& = & {\Vert \left ( {\lambda_{n}}  g \chi_{B_{n}} + (h \circ f^{-k_n}) \chi_{f^{k_{n}}(B_{n})} \right) \circ f^{k_{n}} - h \Vert}_{p}^{p}\\
& = &   {\Vert {\lambda_{n}} \left (g \chi_{B_{n}}  \circ f^{k_n} \right) + \left (h \circ f^{-k_n} \chi_{f^{k_{n}}(B_{n})} \right) \circ f^{k_n} - h \Vert}_{p}^{p} \\
& \leq & 2^{p} \left [{(\lambda_{n})}^{p} {\Vert g \Vert}_{\infty}^{p} \mu(f^{- k_n}(B_{n})) + {{\Vert h \Vert}_{\infty}}^{p} \mu(B \setminus B_{n}) \right ]\\
& = &   2^{p} \left [ {\Vert g \Vert}_{\infty}^{p} {(\lambda_{n})}^{p}\mu(f^{- k_n}(B_{n})) + {{\Vert h \Vert}_{\infty}}^{p} \mu(B \setminus B_{n})\right ]\\
& < & 2^{p} \left [{\Vert g \Vert}_{\infty}^{p}  \epsilon_{n} + {\Vert h \Vert}_{\infty}^{p} \epsilon_{n} \right ], 
\end{eqnarray*}
that is, 
\[\lim_{n \rightarrow \infty} {\Vert  {\lambda_{n}} T_{f}^{k_{n}} v_{n} - h \Vert}_{p}= 0.\]

\noindent
Thus, for $n$ sufficiently large, ${\lambda_{n}}  T_{f}^{k_{n}}(U) \cap V \not=\emptyset$. Hence, we take for a such large $n$, 
$k_{n} \in {\mathbb N}$ and $\lambda_{n} \in {\mathbb R}^{+}$, completing the proof.
\end{proof}

\section{$\mathbb C$-supercyclicity in dissipative setting}
Next, we prove the shift-like characterization, in the dissipative setting, of composition operators, for what concerns supercyclicity.  
\begin{lem} {\em{\cite[Lemma 4.2.3]{DAnielloDarjiMaiuriello}}} \label{factorBw}
Let $T_f$ be a dissipative composition operator of bounded distortion, generated by $W$. Consider the bilateral weighted backward shift $B_w$ with weights
\[w_{k} =  \left( \frac{\mu(f^{k-1}(W))}{\mu(f^{k}(W))}\right)^{\frac{1}{p}}, k \in \Z.  \tag{WS} \] 
Then, $B_w$ is a factor of $T_f$ by the factor map $\Pi: L^p(X) \rightarrow \ell ^p({\mathbb Z}) $ defined as 
 \[\Pi(\varphi)=\left  \{ \dfrac{\mu(f^{k}(W))^{\frac{1}{p}}}{\mu(W)} \int_{W}  \varphi \circ f^{k} d \mu \right \}_{k \in \Z}  \]
\end{lem}

\begin{rmk} \label{rm}
As a consequence of the previous result, we have that if $T_f$ is cyclic, hypercyclic, $\mathbb R$- or $\mathbb C$-supercyclic, then so is the $B_w$ given by \em{(WS)}.
\end{rmk}

We now, in order to proceed, need to recall the composition operator representation of the weighted backward shift $B_w$, more specifically, the following proposition.
 \begin{prop}{\em{\cite[Proposition 2.0.1]{DAnielloDarjiMaiuriello2}}} \label{Tg} Every weighted backward shift $B_w$ is conjugate, by an isometry, to the composition 
 operator $(\Z,{\mathcal P}(\Z), \nu, g, T_g)$,  where \[g(i) = i+1,  \]
\[ \nu(0)=1,  \ \ 
\nu(i)= \begin{cases}
 \dfrac{1}{(w_{1}\cdots w_i)^{p}}, \ \ \ \  \ \  \ \ i > 0 \\
 \left( w_{i+1}\cdots w_0\right)^{p}, \ \ \ \ i < 0.
 \end{cases}\]
 Moreover, for every $i \in \Z$, \[ w_i = \left(\frac{\nu(i-1)}{\nu(i)}\right)^{\frac{1}{p}},\] and when $B_w$ is given by \em{(WS)}, we have \[ \nu(i) =  \dfrac{\mu(f^{i}(W))}{\mu(W)}.\] 
 \end{prop}

 \begin{prop}[$\mathbb R$-supercyclicity Sufficient Condition] \label{newcondition}
Let $(X,{\mathcal B},\mu, f, T_f)$ be a  dissipative composition dynamical system generated by $W$. 
If, for each $\epsilon >0$ and for each $N \in \mathbb N$, there exists $k \geq 1$, $\lambda \in \mathbb R^{+}$ such that 
\[ \mu(f^k(\bigcup_{\vert j \vert \leq N} f^j(W)))<{\epsilon}{\lambda ^p} \ \  \text{ and } \ \ \mu(f^{-k}(\bigcup_{\vert j \vert \leq N}f^j(W)))<{\epsilon}{{\lambda}^{-p}} \]
then $T_f$ is $\mathbb R$-supercyclic and, therefore, $\mathbb C$-supercyclic.
\end{prop}
\begin{proof}
Let $\epsilon >0$ and $B \in \mathcal B$, with $0<\mu(B) < \infty$. Let $N \in \mathbb N$ be so large that \[ \mu\left( B \setminus (\bigcup_{\vert j \vert \leq N} B_j)\right) < \epsilon \]
where $B_j:=  B \cap f^j(W).$ Define 
\[B'=\bigcup_{\vert j \vert \leq N} B_j =\bigcup_{\vert j \vert \leq N} (B \cap f^j(W)).\]
 Then, $\mu(B \setminus B') < \epsilon$. 
By hypothesis, in correspondence of $\epsilon$ and $N$, there exists $k \geq 1$,  $\lambda \in \mathbb R^{+}$ such that 
\[ \mu(f^k(\bigcup_{\vert j \vert \leq N}f^j(W)))<{\epsilon}{\lambda ^p}  \ \  \text{ and } \ \ \mu(f^{-k}(\bigcup_{\vert j \vert \leq N}f^j(W)))<{\epsilon}{\lambda ^{-p}}.\]
Hence
\begin{eqnarray*}
\mu(f^k(B')) & =& \mu(f^k(\bigcup_{\vert j \vert \leq N}  B_j ))=\mu(f^k (\bigcup_{\vert j \vert \leq N} (B \cap f^j(W)))) \\
& \leq & \mu(f^k(\bigcup_{\vert j \vert \leq N}f^j(W))) < {\epsilon}{\lambda ^p}
\end{eqnarray*}
and 
\begin{eqnarray*}
\mu(f^{-k}(B')) & = & \mu(f^{-k}(\bigcup_{j=-N}^N B_j ))=\mu(f^{-k} (\bigcup_{j=-N}^N (B \cap f^j(W))))\\
&  \leq & \mu(f^{-k}(\bigcup_{\vert j \vert \leq N}f^j(W))) <{\epsilon}{\lambda ^{-p}}
\end{eqnarray*}
Therefore, by the arbitrariness of $\epsilon$ and $B$, the thesis follows, that is, for all $\epsilon >0$, for all $B \in \mathcal B$ with $0<\mu(B) < \infty$, there exist $B' \subseteq B$ 
and $k \geq 1$, $\lambda \in \mathbb R^{+}$, such that \[ \mu(B \setminus B')<\epsilon, \,\, \mu(f^{-k}(B'))<{\epsilon}{\lambda ^{-p}},\,\,  \mu(f^{k}(B'))<{\epsilon}{\lambda ^{p}}, \] i.e., 
by Theorem \ref{generalcase}, $T_f$ is $\mathbb R$-supercyclic (and, hence, $\mathbb C$-supercyclic).
\end{proof}

It is already known that, in the dissipative context with the bounded distortion, with respect to some chaotic and to some hyperbolic properties, composition operators 
behave exactly as the specific bilateral weighted backward shift given in (WS) of Lemma \ref{factorBw}. More precisely: many chaotic properties (among which hypercyclicity), 
as well as many hyperbolic properties, can be transferred from $T_f$ to $B_w$ and viceversa. For this reason, in this setting, they are said to have a ``shift-like'' behaviour. 
A detailed explanation of such behavior is provided in \cite{DAnielloDarjiMaiuriello2}. We now prove that the same happens with $\mathbb R$- and $\mathbb C$-supercyclicity.

\begin{thm}[$\mathbb R$-supercyclicity: shift-like behavior] \label{Rlike}
Let $(X,{\mathcal B},\mu, f, T_f)$ be a  dissipative composition dynamical system of bounded distortion generated by $W$.  The composition operator $T_f$ is $\mathbb R$-supercyclic 
if and only if the weighted shift $B_w$ given in (WS) is so.
\end{thm}

\begin{proof}
Since supercyclicity is preserved by semi-conjugation and $B_w$ is a factor of $T_f$,  if $T_f$ is $\mathbb R$-supercyclic, then  $B_w$ is $\mathbb R$-supercyclic.\\
Hence, we only need to prove the converse. To this aim, the sufficient condition given in Proposition \ref{newcondition} will be used.
Assume  $B_w$ $\mathbb R$-supercyclic, meaning that the operator $T_g$ is $\mathbb R$-supercyclic.
Let $\epsilon >0$ and $N \in \N$. Consider $\tilde{\epsilon}=  \dfrac{1}{2} \displaystyle{\min_{\vert j \vert \leq N}}  \left \{\dfrac{\epsilon}{\mu(W)},  \nu(g^{j}(\{0\})) \right \}$ and let  
\[B=\bigcup_{\vert j \vert \leq N}  g^{j}(\{0\}). \]
Clearly, $0 < \nu(B) < \infty$. As $T_g$ is $\mathbb R$-supercyclic, then from Theorem \ref{generalcase} it follows that there exist $B' \subseteq B$, $k \geq 1$, $\lambda \in \mathbb R^+$, such that
 \[ \nu(B \setminus B')<\tilde{\epsilon}, \,\, \nu(g^{k}(B'))<{\tilde{\epsilon}}{ \lambda ^p}, \ \ \nu(g^{-k}(B'))<{\tilde{\epsilon}}{ \lambda^{-p}}. \]
It is not difficult to see that it must be $B' =B$. In fact, assume by contradiction $B' \neq B$, that is, there exists $\tilde{j}$, with $\vert \tilde{j} \vert \leq N$, such that $\tilde{j} \in B \setminus B'$. Then, 
\[  \min_{\vert j \vert \leq N} \{ \nu(g^{j}(\{0\}))\} \leq \nu(\{\tilde{j}\}) \leq \nu(B \setminus B') < \tilde{\epsilon} \]
which is impossible by the choice of $\tilde{\epsilon}$. Hence, it must be $B' = B$. \\
In particular, note that 
 \[\sum_{\vert j \vert \leq N} \nu(g^{k+j}(\{0\}))=\nu(g^{k}(\bigcup_{\vert j \vert \leq N}  g^{j}(\{0\})))=  \nu(g^{k}(B')) < {\tilde{\epsilon}}{\lambda ^p}<\dfrac{\epsilon \lambda ^p}{\mu(W)} \] 
and 
 \[\sum_{\vert j \vert \leq N} \nu(g^{-k+j}(\{0\}))=\nu(g^{-k}(\bigcup_{\vert j \vert \leq N}  g^{j}(\{0\})))=  \nu(g^{-k}(B')) < {\tilde{\epsilon}}{\lambda^{-p}}<\dfrac{\epsilon \lambda ^{-p}}{\mu(W)}. \] 
Then, it follows that
\begin{eqnarray*}
 \mu( f^{k}( \bigcup_{\vert j \vert \leq N} f^{j}(W) )) =\sum_{\vert j \vert \leq N} \frac{\mu(f^{k+j}(W))}{\mu(W)} {\mu(W)} &=& \sum_{\vert j \vert \leq N} \nu(\{k+j\}) {\mu(W)}\\
 &=& \sum_{\vert j \vert \leq N} \nu(g^{k+j}(\{0\})) {\mu(W)} \\
 &<&  \dfrac{\epsilon \lambda^p}{\mu(W)} {\mu(W)} 
  \end{eqnarray*}
and, analogously,
\begin{eqnarray*}
\mu( f^{-k}( \bigcup_{\vert j \vert \leq N} f^{j}(W) ))=\sum_{\vert j \vert \leq N} \frac{\mu(f^{-k+j}(W))}{\mu(W)}  {\mu(W)}  &=& \sum_{\vert j \vert \leq N} \nu(\{-k+j\}) {\mu(W)} \\
&=& \sum_{\vert j \vert \leq N} \nu(g^{-k+j}(\{0\})) {\mu(W)} \\
&<&  \dfrac{\epsilon \lambda^{-p}}{ \mu(W)} {\mu(W)} 
  \end{eqnarray*}
i.e.,
\begin{eqnarray*}
\mu( f^{k}( \bigcup_{\vert j \vert \leq N} f^{j}(W) )) <{\epsilon \lambda^p} \ \ \text{and} \ \ \mu( f^{-k}( \bigcup_{\vert j \vert \leq N} f^{j}(W) )) < {\epsilon \lambda^{-p}}.
 \end{eqnarray*}
By the arbitrariness of $\epsilon$ and $N$, using Proposition \ref{newcondition}, it follows that $T_f$ is ${\mathbb R}$-supercyclic. 
\end{proof}

\begin{cor}[$\mathbb C$-supercyclicity: shift-like behavior] \label{dissipativeC}
Let $(X,{\mathcal B},\mu, f, T_f)$ be a  dissipative composition dynamical system of bounded distortion generated by $W$.  The composition operator $T_f$ is $\mathbb C$-supercyclic if and only if the weighted shift $B_w$ given in (WS) is so.
\end{cor}
\begin{proof}
It is already mentioned in Remark \ref{rm} that $\mathbb C$-supercyclicity can be transferred from $T_f$ to $B_w$. For the other direction, note that $B_w$ is $\mathbb C$-supercyclic if and only if it is 
$\mathbb R$-supercyclic, implying, by Theorem \ref{Rlike}, that $T_f$ is $\mathbb R$-supercyclic and, hence, $\mathbb C$-supercyclic.
\end{proof}
In order to prove the characterization of $\mathbb C$-supercyclicity in the dissipative case, we use the following known result \cite[Corollary 1.39, Remark 1.41]{BayartMatheron}. 

\begin{thm}[Characterization of $\mathbb C$-supercyclicity for $B_{w}$] \label{thmBwC}
Let $B_{w}$ be a bilateral  weighted backward shift on $l^{p}({\mathbb Z})$, $1 \leq p < \infty$, with weight sequence ${\bf w} = (w_{n})_{n \in {\mathbb Z}}$. Then, $B_{w}$ is supercyclic if and only if, for any $q \in {\mathbb N}$, 
\[\lowlim_{n \rightarrow +\infty}  {\left (w_{1} \cdots w_{n+q} \right)}^{-1} \cdot \left (w_{0} \cdots w_{-n+q +1}\right)  =0.\] 
\end{thm}

\begin{thm}[Characterization of $\mathbb C$-supercyclicity - Dissipative case] \label{thmdissipC}
Let $(X,{\mathcal B},\mu, f, T_f)$ be a  dissipative composition dynamical system of bounded distortion generated by $W$.  The composition operator $T_f$ is $\mathbb C$-supercyclic if and only if, for any $q \in \N$, 
\[\lowlim_{n \rightarrow +\infty} \left ( \mu(f^{q-n}(W)) \cdot  \mu(f^{q+n}(W)) \right )=0.\] 
\end{thm}
\begin{proof}
By Corollary \ref{dissipativeC},  the operator $T_f$ is $\mathbb C$-supercyclic if and only if the weighted shift $B_w$, with weights $w_{k} =  \left( \frac{\mu(f^{k-1}(W))}{\mu(f^{k}(W))}\right)^{\frac{1}{p}}$, $k \in \Z,$ is so. 
Then, the thesis follows from Theorem \ref{thmBwC}.
\end{proof}

\begin{rmk} As pointed out in \cite[Remark page 20]{BayartMatheron}, when the shift $B_w$ is invertible, the characterization in Theorem \ref{thmBwC} can be stated in
a simpler way: $B_{w}$ is supercyclic if and only if
\[\lowlim_{n \rightarrow +\infty}  {\left (w_{1} \cdots w_{n} \right)}^{-1} \cdot \left (w_{-1} \cdots w_{-n}\right)  =0.\] 
Indeed, since the weights $w_n$ are bounded above and below, the products ${(w_{1}  \cdots w_{n+q})}^{-1}$ and $(w_{0} \cdots w_{-n+q+1})$ are equivalent to ${(w_1 \cdots w_{n})}^{-1}$  and $(w_{-1} \cdots  w_{-n})$, respectively, 
up to constants depending only on $q$.\\
Hence, also, the characterization in Theorem \ref{thmdissipC} can be stated in
a simpler way: 
$T_f$ is $\mathbb C$-supercyclic if and only if
\[\lowlim_{n \rightarrow +\infty} \left ( \mu(f^{-n}(W)) \cdot  \mu(f^{n}(W)) \right )=0 \tag{SC} .\] 
\end{rmk}

\begin{exmp}
Let $X = {\mathbb R}$, ${\mathcal B}$ the collection of Borel subsets of ${\mathbb R}$, and $f(x) = x+1$. For any measure $\mu$ on  ${\mathbb R}$ with $0<\mu(W) < \infty$, $W=[0,1),$ we get a  dissipative system generated by $W$. As in \cite{DAnielloDarjiMaiuriello}, take $\mu$ to be given by a density, i.e.,
\[\mu(B) = \int _B h d\lambda,\]
where $\lambda$ is the Lebesgue measure on ${\mathbb R}$ and $h$ is some non-negative Lebesgue measurable function. 
As \[ \frac{d (\mu f^k)}{d\lambda} = \frac{d (\mu f^k)}{d\mu} \cdot  \frac{d\mu }{d\lambda} 
\]
and 
\[ \frac{d (\mu f^k)}{d\lambda} (x) = h(x+k), \ \ \ \ \ \ \  \ \ \ \ \ \frac{d\mu }{d\lambda} (x) = h(x),
\]
we have that \[ \frac{d (\mu f^k)}{d\mu} (x) = \frac{h(x+k)}{h(x)}.
\]
Let, now, $\mu$ be the measure whose density is $h(x)$ so defined 

\[h(x)= \left\{ \begin{array}{ll}
e^{2x}  & \mbox{ if } x \in ]-\infty, 0]\\
e^{x}   & \mbox{ if } x \in [0, + \infty[ \\
\end{array}
\right.\] 
Hence, if $x \in [0, 1[$, then 
 \[\frac{d (\mu f^k)}{d\mu} (x)= \left\{ \begin{array}{ll}
\frac{e^{x+k}}{e^{x}} = e^{k}  & \mbox{ if }  k \in {\mathbb N}\\
\frac{e^{2(x+k)}}{e^x}= e^{x+2k}     & \mbox{ if } -k \in {\mathbb N}\\
\end{array}
\right.\] 
Let $\rho_k= \frac{d (\mu f^k)}{d\mu}$, $m_{k} = \underset{ x \in W}{\mathrm{ess\,inf}} \  {\rho}_{k} (x)
$ and $M_{k} = \underset{ x \in W}{\mathrm{ess\,sup}} \  {\rho}_{k} (x) $. Then
\[ \frac{M_{k}}{m_{k}} = \left\{ \begin{array}{ll}
1 & \mbox{ if }  k \geq 0\\
e    & \mbox{ if } k \leq  -1\\
\end{array}
\right.\] 
Hence, the sequence $\{\frac{M_k}{m_k}\}_{k \in {\mathbb Z}}$ is bounded.  By \cite[Proposition 2.6.6]{DAnielloDarjiMaiuriello}, the above dissipative system satisfies the bounded distortion property. 
Notice that
\[ \mu(W)= e - 1,\]
and, for each $n \geq 1$, 
\[ \mu(f^n(W))= \int _n^{n+1} h(x) dx=  \int _n^{n+1} e^{x} dx =e^{n}(e -1)\]
and
\[ \mu(f^{-n}(W))= \int _{-n}^{-n+1} h(x) dx= \int _{-n}^{-n+1} e^{2x} dx =\frac{1}{2}e^{-2n}(e^{2}-1).\]
Then, for each $n \in {\mathbb N}$, 
\[0 \leq \mu(f^{-n}(W)) \cdot \mu(f^{n}(W))  \leq \frac{1}{2}(e+1){(e-1)}^{2}  e^{-n} \] 
 implying, by {\em{(SC)}}, that $T_f$ is $\mathbb C$-supercyclic.
\end{exmp}

{\begin{small}
{\bf Acknowledgments.} {The authors wish to thank the referee for the careful reading of the paper and the useful suggestions. This research has been partially supported by the following projects. The project COSYMA: COmplex SYstem MAintenance; the INdAM group GNAMPA ``Gruppo Nazionale per l'Analisi Matematica, la Probabilit\'a e le loro Applicazioni" within the project 2024 DYNAMIChE: DYNAmical Methods: Inverse problems, Chaos and Evolution; the project PRIN 2022 QNT4GREEN: Quantitative Approaches for Green Bond Market: Risk Assessment, Agency Problems and Policy Incentives. It has been partially accomplished within the group UMI-TAA: Approximation Theory and Applications.}
\end{small}}

\bibliographystyle{siam}
\bibliography{biblio}

\Addresses

\end{document}